\documentclass[12pt, reqno, a4paper]{amsart}
\usepackage{graphicx, epsfig, psfrag}
\usepackage{amsfonts,amsmath,amssymb,amsbsy,amsthm}
\usepackage{bm}
\usepackage{color}
\usepackage{float}
\usepackage{hyperref}
\usepackage{mathrsfs}
\usepackage{subfigure}
\usepackage{bbm}
\usepackage{longtable}

\usepackage{pdfsync}

\usepackage{environments}

\def\ilist{\renewcommand{\labelenumi}{(\roman{enumi})}}

\definecolor{cocol}{rgb}{0.7, 0, 0}

\definecolor{escol}{rgb}{0, 0.7, 0}

\renewcommand{\cases}[1]{\left\{ \begin{array}{rl} #1 \end{array} \right.}
\newcommand{\smfrac}[2]{{\textstyle \frac{#1}{#2}}}

\def\Xint#1{\mathchoice
{\XXint\displaystyle\textstyle{#1}}%
{\XXint\textstyle\scriptstyle{#1}}%
{\XXint\scriptstyle\scriptscriptstyle{#1}}%
{\XXint\scriptscriptstyle\scriptscriptstyle{#1}}%
\!\int}
\def\XXint#1#2#3{{\setbox0=\hbox{$#1{#2#3}{\int}$ }
\vcenter{\hbox{$#2#3$ }}\kern-.6\wd0}}
\def\mint{\Xint-}

\def\b{\big}
\def\B{\Big}
\def\bg{\bigg}

\def\sep{\,|\,}
\def\bsep{\,\b|\,}

\def\R{\mathbb{R}}
\def\N{\mathbb{N}}
\def\Z{\mathbb{Z}}

\def\dx{\,{\rm d}x}

\def\dz{\,{\rm d}z}

\def\dt{\,{\rm d}t}

\def\dd{\,{\rm d}}

\def\dk{\,{\rm d}k}

\def\<{\langle}
\def\>{\rangle}

\def\mA{{\sf A}}

\def\mC{{\sf C}}

\def\mG{{\sf G}}

\def\D{\partial}

\def\tC{\mathbb{C}}

\def\Div{{\rm div}\,}
\def\loc{{\rm loc}}

\def\DD{\mathscr{D}}
\def\DDe{\dot{\DD}}
\newcommand{\We}[1]{\dot{W}^{#1}}

\newcommand{\He}[1]{\dot{H}^{#1}}

\begin{document}

\title{A Note on Linear Elliptic Systems on $\R^d$}

\author{Christoph Ortner}
\address{Christoph Ortner\\ Mathematics Institute\\ Zeeman Building\\ 
  University of Warwick\\ Coventry CV4 7AL\\ UK}
\email{christoph.ortner@warwick.ac.uk}

\author{Endre S\"{u}li}
\address{Endre S\"{u}li\\ Mathematical Institute\\ University of Oxford\\
  24--29 St Giles'\\ Oxford OX1 3LB\\ UK}
\email{endre.suli@maths.ox.ac.uk}

\date{\today}




\maketitle

\section{Introduction}
\label{sec:intro}
We are concerned with the well-posedness of linear elliptic systems of
the form
\begin{align}
  \label{eq:ellsys_strong}
  -\Div \tC : \D u =~& f, \\
  \label{eq:farfield_bc_formal}
  u(x) \sim~& 0, \quad \text{as } |x| \to \infty,
\end{align}
where $\tC \in C(\R^d; \R^{m^2d^2})$ is bounded
and satisfies the Legendre--Hadamard condition,
\begin{displaymath}
  \tC_{i\alpha}^{j\beta}(x) v_i v_j k_\alpha k_\beta \geq c_0 |v|^2
  |k|^2
  \qquad \forall x, k \in \R^d, \quad v \in \R^m.
\end{displaymath}
The functions $f, u : \R^d \to \R^m$ (we will define the precise
function spaces to which $f$ and $u$ belong to later on), and $\tC :
\mG = (\tC_{i\alpha}^{j\beta} \mG_{j\beta})_{i\alpha}$ denotes the
contraction operator.

The concrete problem of interest, for which we require this theory,
arises from the linearization of the equations of anisotropic finite
elasticity in infinite crystals, however, our results are more
generally applicable to translation-invariant problem posed on
$\R^d$. Some of the main challenges to be overcome in
translation-invariant problems on infinite domains are the absence of
Poincar\'e-type inequalities, and the interpretation of boundary
conditions.

A common approach to PDEs on infinite domain, as well as for exterior
problems, is the formulation in weighted function spaces (see, e.g.,
\cite{Kufner, WaJaBe:1985}).  Our aim in this note is to outline a
more straightforward existence, uniqueness, and regularity theory in
Sobolev spaces of Beppo Levi type (also called homogeneous Sobolev
spaces). Such spaces have previously been analyzed in detail in
\cite{DL} and used for the solution of elliptic PDEs (see, e.g.,
\cite{SamVar:2003, Sohr, GalSim:1990, KozSoh:1991}).

In the present work we describe a version of the homogeneous Sobolev
space approach. Variants (and sometimes generalisations) of most of
our results can be found in the cited literature; however, the
equivalence class viewpoint considered here is not normally taken and
the growth characterisation given in Theorem \ref{th:growth} appears
to be new. This research note is intended as an elementary
introduction to and reference for some key ideas.

We wish to define the homogeneous Sobolev space as a closure of smooth
functions with compact support. The following cautionary example was
discussed by Deny \& Lions \cite{DL}: let $u_n : \R \to \R$ be defined
by
\begin{displaymath}
  u_n(x) := n \max(0, 1 - |x| / n^3),
\end{displaymath}
where $u_n$ has compact support and $u_n'(x) = \pm \frac1{n^2}$ in
$\pm (0, n^3)$, and hence $\|\D u_n \|_{L^2} \to 0$ as $n \to
\infty$. However, $u_n$ clearly does not converge in the topology of
$\DD'$.

To avoid this difficulty, we will define spaces of equivalence
classes, or, factor spaces.  Indeed, if we shift $u_n$ to obtain $v_n
:= u_n - n$, then it is straightforward to see that $v_n \to 0$ in the
sense of distributions, which is consistent with the convergence $\|\D
u_n \|_{L^2} = \|\D v_n\|_{L^2} \to 0$ as $n\to \infty$.

\subsection{Notation}

$B_R$ denotes the open ball, centre $0$, radius $R$ in $\R^d$, $d \in \{1,2,\dots\}$; 
$p \in [1,\infty]$, $p'=p/(p-1)$, and $p^*$ denotes the Sobolev
conjugate of $p \in [1,d)$, $d>1$, defined by
$1/p^*= 1/p - 1/d$. For Lebesgue and Sobolev spaces of functions defined on the whole of $\R^d$ we shall suppress the symbol $\R^d$
in our notations for these function spaces, and will simply write $L^p$ and $W^{1,p}$, respectively, instead of
$L^p(\R^n)$ and $W^{1,p}(\R^n)$. We define the integral average $(u)_A$ of a locally integrable function $u \in L^1_{\rm loc}$
over a measurable set $A \subset \R^n$, $|A|:=\mbox{meas}(A)<\infty$, by  $(u)_A:=|A|^{-1} \int_A u(x) \dd x$. Throughout this
note $\int$ will signify $\int_{\mathbb{R}^d}$.

Assuming that $\Omega_k$, $k=1, 2, \dots$, in an increasing sequence of bounded open
sets in $\R^n$, $L^p_{\rm loc}(\Omega)$ is equipped with the family of seminorms
\[ \|u\|_{L^p(\Omega_k)} := \left(\int_{\Omega_k} |u(x)|^p \dd x\right)^{\frac{1}{p}}.\]
The linear space $L^p_{\rm loc}(\Omega)$ is then a Fr\'echet space (i.e., a metrizable
and complete topological vector space).

\section{Sobolev spaces of equivalence classes}

For any measurable function $u : \R^d \to \R$, let $[u] := \{ u + c
\sep c \in R \}$ denote the equivalence class of all translations of $u$. Let
$\DD$ denote the space of test functions ($C^\infty$ functions with
compact support in $\R^d$), and let $\DDe := \{ [u] \sep u \in \DD \}$ be the
associated linear space of equivalence classes $[u]$ of translations of $u \in \DD$.

We denote the linear space of equivalence classes $[u]$ of functions $u \in W^{1,p}_{\rm loc}$
with $p$-integrable gradient by
\begin{displaymath}
  \We{1,p} := \b\{ [u] \bsep u \in W^{1,p}_\loc, ~ \D u \in L^p \b\},
\end{displaymath}
equipped with the norm
\begin{displaymath}
  \| [u] \|_{\We{1,p}} := | u |_{W^{1,p}} = \| \D u \|_{L^p},\quad u \in [u].
\end{displaymath}

\begin{proposition}
  $\We{1,p}$ is a Banach space.
\end{proposition}
\begin{proof}
  It is clear that $\|\bullet\|_{\We{1,p}}$ is a semi-norm on $\We{1,p}$. To show
  that it is a norm, suppose that $\|[u]\|_{\We{1,p}} = 0$. Then $\D u
  = 0$ and hence $u$ is a constant, that is, $[u] = [0]$.

  To prove that $\We{1,p}$ is complete, suppose that $([u_j])_{j \in
    \N}$ is a Cauchy sequence. Let $u_j \in [u_j]$ be defined through
  the condition that $(u_j)_{B_1} = 0$. Then, it is straightforward to
  show that there exist $u \in L^p_\loc$ and $g \in L^p$ such that
  $u_j \to u$ in $L^p_\loc$ and $\D u_j \to g$ in $L^p$. By the uniqueness of
  the distributional limit, it then follows that $g = \D u$. Hence we have shown that there exists $u \in
  W^{1,p}_\loc$ such that $\D u_j \to \D u$ in $L^p$, that is, $[u_j]
  \to [u]$ in $\We{1,p}$ as $j \to \infty$.
\end{proof}

The next result establishes that test functions are dense in
$\We{1,p}$. This result is a special case of \cite[Thm. 1]{Sohr}.

\begin{theorem}
  Let $p \in (1, \infty)$ or $p = 1$ and $d > 1$; then, $\DDe$ is
  dense in $\We{1,p}$. 
\end{theorem}
\begin{proof}
  Suppose first that $d > 1$. Fix $u \in [u] \in \We{1,p}$. Since
  $\DD$ is dense in $W^{1,p}$ it is sufficient to show the existence
  of a sequence $(u_n) \subset W^{1,p}$ such that $[u_n] \to [u]$ in
  $\We{1,p}$.

  Let $\eta \in C^1([0, \infty))$ be a cut-off function satisfying
  \begin{displaymath}
    \eta(r) = \cases{1, & r \leq 1, \\ 0, & r \geq 2.}
  \end{displaymath}
  For each $n \in \N$, let $A_n := B_{2n} \setminus B_n$ and define
  \begin{displaymath}
    u_n(x) := \eta(|x|/n)\, \b(u(x) - (u)_{A_n}\b).
  \end{displaymath}
  Hence,
  \begin{displaymath}
    \D u_n(x) = n^{-1}\, \eta'(|x|/n)\, \smfrac{x}{|x|}\, \b(u - (u)_{A_n}\b) +
    \eta(|x|/n)\, \D u.
  \end{displaymath}
  Since $u \in W^{1,p}_\loc$ and $u_n$ has compact support, it is
  clear that $u_n \in W^{1,p}$. Further, since $\eta'$ is uniformly bounded, we can estimate
  \begin{align*}
    \| \D u - \D u_n \|_{L^p}
    \leq~& \b\| n^{-1} \eta' (u - (u)_{A_n}) \b\|_{L^p}
    + \b\| (1 - \eta) \D u \|_{L^p} \\
    \leq~& C n^{-1} \| u - (u)_{A_n} \|_{L^p(A_n)}
    + \| \D u \|_{L^p(\R^d \setminus B_n)}.
  \end{align*}
  Poincar\'{e}'s inequality on $A_1$ and a standard scaling argument
  then imply that
  \begin{displaymath}
    C n^{-1} \| u - (u)_{A_n} \|_{L^p(A_n)} \leq (C n^{-1}) (C_P n) \|
    \D u \|_{L^p(A_n)} \leq C \| \D u \|_{L^p(\R^d \setminus B_n)},
  \end{displaymath}
  that is, $\| \D u - \D u_n \|_{L^p} \leq C \| \D u \|_{L^p(\R^d
    \setminus B_n)}$. Since $\|\D u \|_{L^p}$ is finite it follows
  that this upper bound tends to zero as $n \to \infty$.

  Hence, we have constructed a sequence $(u_n) \subset W^{1,p}$ such
  that $\D u_n \to \D u$ in $L^p$, or, equivalently $[u_n] \to [u]$ in
  $\We{1,p}$.

  If $d = 1$, then $A_n$ is not simply connected and hence the
  Poincar\'e inequality does not hold. Instead, we prove that for any
  $u \in W^{1,p}_{\rm loc}$ with $u' = \chi_{(a, b)}$ (the
  characteristic function of an interval) we can construct a sequence
  $[u_n] \in \DDe$ approching $[u]$. Density of the span of
  characteristic functions in $L^p$ then implies the stated result for
  $d = 1$. Let $u_n$ be defined by
  \begin{displaymath}
    u_n'(x) = \cases{ 1, & x \in (a, b), \\
      -1/n , & x \in (b, b+ n(b-a)), \\
      0, & \text{otherwise},}
  \end{displaymath}
  then it is a straightforward computation to show that $u_n' \to u' =
  \chi_{a,b}$ in $L^p$ for any $p > 1$, but not in $L^1$.
\end{proof}

\begin{remark}
  If $d = 1$ then $\DDe$ is not dense in $\We{1,1}$. If this were the
  case, then all functions $u \in \We{1,1}$ would satisfy $\int_\R u'
  \dx = 0$. However, it is clear that the equivalence class of the
  function $u(x) = \max(0, \min(x, 1))$ belongs to $\We{1,1}$, but
  does not satisfy this condition.
\end{remark}

Our next result classifies the growth or decay of classes $[u] \in
\We{1,p}$ at infinity. Case (i) is essentially contained in
\cite[Prop. 2.4(i)]{KozSoh:1991}; cases (ii) and (iii) are new to the
best of our knowledge.

\begin{theorem}
  \label{th:growth}
  There exist linear maps $J_\infty : \We{1,p} \to C^\infty$ and $J_0
  : \We{1,p} \to W^{1,p}$ such that
  \[[u] = [J_\infty[u] + J_0[u]], \qquad \text{for } [u] \in \We{1,p},\]
  and
  \begin{align*}
    \| \D J_\infty [u] \|_{L^p} \leq \|\D u \|_{L^p},
    \quad \| \D J_\infty[u] \|_{L^\infty} \leq  \| \D u \|_{L^p} \quad
    \text{and}
    \quad \| J_0[u] \|_{W^{1,p}} \leq C \| \D u \|_{L^p},
  \end{align*}
  where $C = C(d) > 0$.

  Moreover, $J_\infty$ may be chosen to satisfy the following growth
  conditions at infinity:
  \begin{enumerate} \ilist
  \item If $p < d$, then $\We{1,p}$ is continuously embedded in
    $L^{p*}$, in the sense  that, for each $[u] \in \We{1,p}$ there exists a unique $u_0 \in
  [u]$ such that $u_0 \in L^{p*}$ and $\| u_0 \|_{L^{p*}} \leq C \| \D u_0 \|_{L^p}$, where
  $C$ is a positive constant independent of $u_0$.
     In particular, $J_\infty[u](x) \to 0$ as $|x| \to \infty$, $x \in \R^d$.
    \medskip
  \item If $p > d$, then $|J_\infty[u](x)| \leq C \| [u] \|_{\We{1,p}}
    |x|^{1/p'}$, $x \in \R^d$.
    \medskip
  \item If $p = d$, then $|J_\infty[u](x)| \leq C \|[u] \|_{\We{1,p}}
    \log (2+ |x|)$, $x \in \R^d$.
 \end{enumerate}
\end{theorem}
\begin{proof}
  We shall assume throughout that $1\leq p < \infty$; in case (ii) the
  choice of $p = \infty$ can be dealt with separately using an
  analogous argument to the one for $d<p<\infty$.

  Let $\eta \in \DD$, $0 \leq \eta \leq 1$, $\int \eta(x) \dx = 1$, fix
  $u \in [u] \in \We{1,p}$ and define
  \begin{displaymath}
    v := \eta \ast u  \in C^\infty, \quad \text{and} \quad
    w := u - v.
  \end{displaymath}
  By Young's inequality for convolutions, $\| \D v \|_{L^p} \leq \| \D
  u \|_{L^p}$, and, because of the assumption that $\eta \leq 1$, it
  is also straightforward to show that $\| \D v \|_{L^\infty} \leq \|
  \D u \|_{L^p}$:
  \begin{align*}
    \b| \D v(x) \b| = \bg| \int \eta(x - z) \D u(z) \dz \bg| \leq
    \bg( \int \eta(x-z) |\D u(z)|^p \dz \bg)^{1/p}
    \leq \| \D u \|_{L^p}.
  \end{align*}

  Next, we show that $w \in W^{1,p}$. It follows directly from the
  definition of $w$ that $\D w = \D u - \D v \in L^p$. Hence, $\| \D w
  \|_{L^p} \leq \| \D u \|_{L^p} + \| \D v \|_{L^p} \leq 2 \| \D u
  \|_{L^p}$. To show that $w \in L^p$, let $R > 0$ be such that ${\rm
    supp}\,\eta \subset B_R$. For any $\xi \in \R^d$ we have
  \begin{align*}
    \int_{B_R(\xi)} |w(x)|^p \dx =~& \int_{B_R(\xi)} \bg| \int \eta(x-z) (u(z) - u(x)) \dz
    \bg|^p \dx \\
    \leq~& \int_{B_R(\xi)} \int \eta(x-z) \b| u(z) - u(x) \b|^p \dz
    \dx \\
    \leq~& \int_{B_{2R}(\xi)} \int_{B_{2R}(\xi)} \b| u(z) - u(x) \b|^p
    \dz \dx \leq C(R) \| \D u \|_{L^p(B_{2R}(\xi))}^p,
  \end{align*}
  where the last inequality is an immediate consequence of
  Poincar\'{e}'s inequality on the ball $B_{2R}(\xi)$. We can cover
  $\R^d$ with countably many balls $B_R(\xi)$, $\xi \in R \Z^d$, such
  that the balls $B_{2R}(\xi)$ have finite overlap, that is, any $x
  \in \R^d$ belongs to at most $m$ balls where $m$ is independent of
  $x$. Summing over all balls gives the result that $\| w \|_{L^p}
  \leq C \| \D u \|_{L^p}$, where $C$ may depend on the support of
  $\eta$ and hence on the dimension $d$, but is independent of the value of $p$.

\smallskip

We now distinguish between three cases, depending on the values of $p$
and $d$. 

{\it (i) $p < d$: } Since $\DDe$ is dense in $\We{1,p}$ and by the
Gagliardo--Nirenberg--Sobolev Inequality, it follows that $\We{1,p}$
is embedded in $L^{p*}$ in the sense that for each $[u] \in \We{1,p}$
there exists a unique $u_0 \in [u]$ such that $u_0 \in L^{p*}$ and $\|
u_0 \|_{L^{p*}} \leq C_{\rm GNS} \| \D u_0 \|_{L^p}$, where $C_{\rm
  GNS}$ is the constant in the Gagliardo--Nirenberg--Sobolev
Inequality. We define
\[J_\infty[u] := v_0 := \eta \ast u_0.\]
It is an immediate consequence of this definition that $v_0 \in
L^{p*}$.  Since $\| \D v_0 \|_{L^\infty}$ is finite, it follows also
that the sequence $( \| v_0 \|_{L^{\infty}(B_1(\xi))})_{\xi \in \Z^d}$
belongs to $\ell^{p*}(\Z^d)$, and this implies that $\| v_0
\|_{L^{\infty}(B_1(\xi))} \to 0$ uniformly as $|\xi| \to \infty$.  We
obtain statement (i) as a special case.

{\it (ii) $p > d$: } In this case we define $J_\infty[u](x) := (\eta
\ast u)(x) - (\eta\ast u)(0)$ for any element $u \in [u]$, which does
not of course change the foregoing results. If we define $v := \eta
\ast u$, then we obtain
\begin{align*}
  |J_\infty[u](r)| =~& |v(r) - v(0)| \leq \bg|\int_0^{|r|} \D v\b(t
  \smfrac{r}{|r|}\b) \smfrac{r}{|r|} \dt \bg| \\
  \leq~& |r|^{1/p'} \bg(\int_{0}^{|r|} \b| \D v\b(t \smfrac{r}{|r|}
  \b)\b|^p \dt\bg)^{1/p} \\
  \leq~& |r|^{1/p'} \bg( \int_0^{|r|} \bg| \int \eta\b(t
  \smfrac{r}{|r|} - z\b) \D u(z) \dz\bg|^p  \dt \bg)^{1/p} \\
  \leq~& |r|^{1/p'} \bg( \int \bg[ \int_0^{|r|} \eta\b(t
  \smfrac{r}{|r|} - z\b) \dt \bg] |\D u(z)|^p \dz \bg)^{1/p}.
\end{align*}
Since the diameter of the support of $\eta$ is independent of $|r|$ it
follows that
\begin{displaymath}
  \int_0^{|r|} \eta\b(t \smfrac{r}{|r|} - z\b) \dt
  \leq C
\end{displaymath}
for some universal constant $C$, which implies {\it (ii)}.

{\it (iii) $p = d$: } In this critical case, we use the fact that
$\We{1,p}$ may be embedded in the space BMO of functions of bounded
mean oscillation (see \cite{JN}), though for simplicity we will not
refer to BMO directly.

If $Q$ is a cube in $\R^d$ with arbitrary orientation and $u \in [u]
\in \We{1,p}$, then
\begin{displaymath}
  \mint_Q \b| u - (u)_Q \b| \dx \leq |Q|^{1/p'-1} \| u - (u)_Q
  \|_{L^p(Q)} \leq C \| \D u \|_{L^p(Q)} \leq C \|[u]\|_{\We{1,p}},
\end{displaymath}
where the second inequality follows on noting that $|Q|^{1/p'-1} =
|Q|^{-1/d} \leq ({\rm diam} Q)^{-1}$.

The key observation is that if $u \in W^{1,p}_\loc$ with $\D u \in
L^p$, then for each $x \in \R^d$ there exist unit cubes $Q_x$ centred
at $x$ and $Q_0$ centred at $0$, such that
\begin{equation}
  \label{eq:bmo_argument}
  \b| (u)_{Q_x} - (u)_{Q_0} \b| \leq C \| \D u \|_{L^p} \log(2+ |x|).
\end{equation}
The proof of inequality \eqref{eq:bmo_argument} will be given below,
after the end of the proof of this theorem.  With this in hand, we can
define $J_\infty[u](x) := v_0 := (\eta \ast u)(x) - (\eta \ast u)(0)$
to obtain
\begin{align*}
  \b| v_0(x) \b| \leq~& \b| v_0(x) - (v_0)_{Q_x} \b| + \b| (v_0)_{Q_0}
  \b|+ \b| (v_0)_{Q_x}
  - (v_0)_{Q_0} \b| \\
  \leq~& C\, \| \D v_0 \|_{L^\infty} + C\, \| \D v_0  \|_{L^\infty}
  + C \,\| \D v_0 \|_{L^p} \log (2+|x|)
  \\
  \leq~& C\, \| \D u \|_{L^p} \log(2+ |x|),
\end{align*}
for a generic (dimension-dependent) constants $C$.  This concludes the
proof of case (iii).
\end{proof}

\begin{figure}
  \includegraphics[width=8cm]{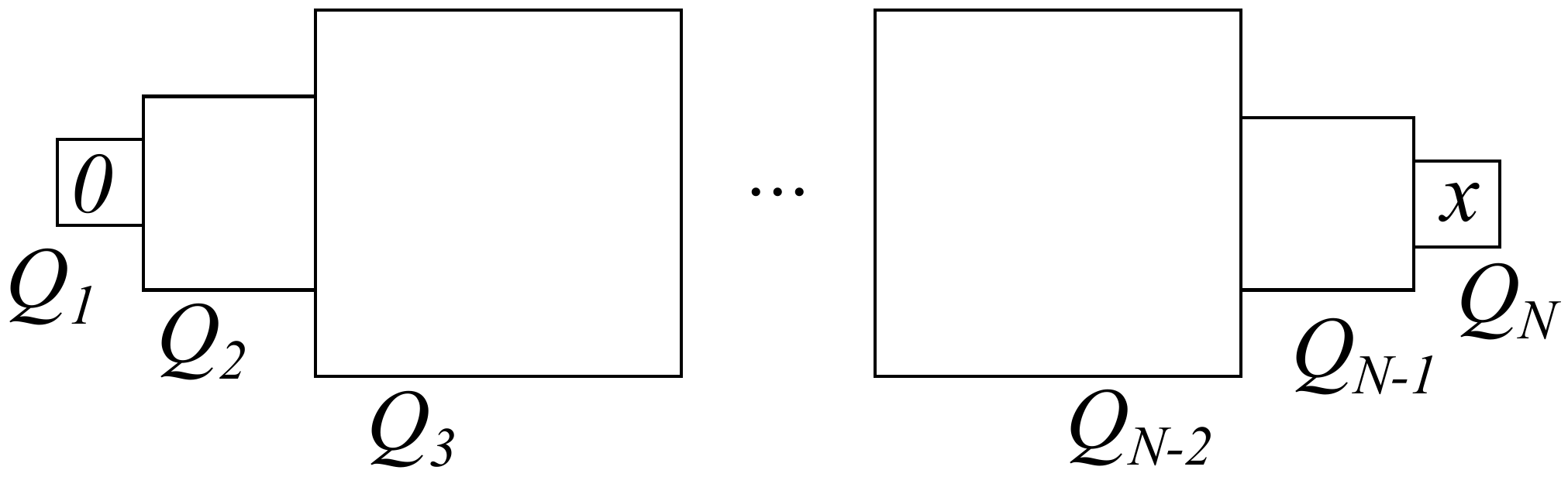}
  \caption{\label{fig:bmo_argument} Visualization of an argument used
    in the proof of \eqref{eq:bmo_argument}.}
\end{figure}

\begin{proof}[Proof of \eqref{eq:bmo_argument}]
  We assume without loss of generality that $|x| \geq 1$.  Let $Q_0,
  Q_x$ be unit cubes centred, respectively, at $0$ and $x$ such that
  one set of edges of each of the cubes $Q_0$ and $Q_x$ is aligned
  with the direction $\vec{0x}$. There
  exists $N \leq C (2+\log|x|)$ and cubes $Q_2, \dots, Q_{N-1}$ with
  the same alignment as $Q_0, Q_x$ and with disjoint interior such
  that, for any two neighbouring cubes, their sidelengths differ by at
  most a factor $2$ and one face of the smaller cube is contained
  within one face of the large cube. See Figure \ref{fig:bmo_argument}
  for a visualization of this argument.

  For any two neighbouring cubes $Q_j, Q_{j+1}$ we have
  \begin{displaymath}
    \b| (u)_{Q_j} - (u)_{Q_{j+1}} \b| \leq C \| \D u \|_{L^p},
  \end{displaymath}
  which is a special case of \cite[Lemma 2]{JN}, but can
  also be verified directly by enclosing $Q_j, Q_{j+1}$ in a larger
  cube of approximately the same size. Hence, defining $Q_x = Q_N$, we
  obtain
  \begin{displaymath}
    \b| (u)_{Q_x} - (u)_{Q_0} \b| \leq \sum_{j = 0}^{N-1} \b|
    (u)_{Q_{j+1}} - (u)_{Q_j} \b| \leq N \| \D u \|_{L^p}. \qedhere
  \end{displaymath}
\end{proof}

\begin{remark}
  The map $J' := J_\infty + J_0$ defines an embedding of $\We{1,p}$
  into $\DD'$. Since $J_0$ is continuous to $W^{1,p}$ and $J_\infty$
  is continuous to $W^{1,\infty}_\loc$ it follows that the embedding
  $J'$ of $\We{1,p}$ into $\DDe$ is in fact continuous. However, it
  is not particularly useful for our purposes since we are explicitly
  interested in operations that are translation invariant, that is,
  independent of the representative $u \in [u]$, whenever $[u] \in
  \We{1,p}$.
\end{remark}

\section{Well-posedness and regularity}
From now on we restrict our presentation to the case $p = 2$ and hence
define $\He{1} := \We{1,2}$. Since we will take particular care that
all operators, linear functionals, and bilinear forms we consider are
translation invariant, we will drop the brackets in $[u] \in \He{1}$
and instead write simply $u \in \He{1}$ instead, by which we mean an
arbitrary representative from the class $[u]$. (For convenience one
may take $u = J_\infty[u] + J_0[u]$.)

Since we consider elliptic systems, we will from now on identify all
function spaces with spaces of vector-valued functions, that is, $L^p
= (L^p)^m$, $\He{1} = (\He{1})^m$, and so forth, for some fixed $m \in
\N$.

Before we embark on the analysis of the elliptic system
\eqref{eq:ellsys_strong} we briefly discuss admissible right-hand
sides $f$ for \eqref{eq:ellsys_strong} as well as the far-field
boundary condition \eqref{eq:farfield_bc_formal}.

\subsection{The dual of $\He{1}$}
\label{sec:dual}
We denote the topological dual of $\He{1}$ by $\He{-1}$. Since
$\He{1}$ is a Hilbert space with inner product $(\D \cdot, \D
\cdot)_{L^2}$ it follows that, for each $\ell \in \He{-1}$, there
exists $F \in L^2$ such that $\ell = -{\rm div} F$ in the
distributional sense. (For a generalisation of this result to
$\We{-1,p}$, $p \in (1,\infty)$ see \cite[Lemma 2.2]{KozSoh:1991}.)

If we wish to define $\ell$ via an $L^2$-pairing, then the following
two examples give concrete conditions:
\begin{enumerate}
\item Let $f \in L^1_\loc$; then we can define $\ell : \DDe \to \R$ by
  \begin{equation}
    \label{eq:linear_fcnl_ex1}
    \ell([u]) := \int_{\R^d} f \cdot u^* \dx, \quad \text{where } u^*
    \in [u], u^* \in \DD.
  \end{equation}
  If, moreover, $f = {\rm div} g$, where $g \in L^2$, then
  \eqref{eq:linear_fcnl_ex1} can be extended to a bounded linear
  functional on $\He{1}$.

\item More concretely, if $f \in L^1$ with $\int_{\R^d} f \dx = 0$ and
  ${\rm div} g = f$, $g \in L^2$, then we may define
  \begin{equation}
    \label{eq:linear_fcnl_ex2}
    \ell([u]) := \int_{\R^d} f \cdot u \dx, \quad \text{for any } u
    \in [u] \in \DDe.
  \end{equation}
  Again, \eqref{eq:linear_fcnl_ex2} can be extended to a bounded
  linear functional on $\He{1}$.

\item Even more concretely, if $f \in L^1 \cap L^\infty$ with
  $\int_{\R^d} f \dx = 0$, and $x \otimes f \in L^1$, then this is
  sufficient to ensure that $\ell$ defined through
  \eqref{eq:linear_fcnl_ex2} can be extended to a bounded linear
  functional on $\He{1}$ (see Lemma \ref{th:linear_fcnl_ex3}
  below). We note, however, that right-hand sides with such strong
  decay assumptions may be more naturally treated within the framework
  of weighted Sobolev spaces \cite{Kufner}.
\end{enumerate}

\begin{lemma}
  \label{th:linear_fcnl_ex3}
  Suppose that $f \in L^1 \cap L^2$ with $\int_{\R^d} f \dx = 0$,
  $f \otimes x \in L^1$,  and let $\ell : \DDe \to \R$ be defined
  through \eqref{eq:linear_fcnl_ex2}; then
  \begin{displaymath}
    \ell(u) \leq C \| \D u \|_{L^2} \qquad \forall u \in \DDe.
  \end{displaymath}
\end{lemma}
\begin{proof}
  Consider the Fourier transform of $f$, which is defined in a
  pointwise sense since $f \in L^1$:
  \begin{displaymath}
    \hat{f}(k) = \int f(x) \exp(-i k \cdot x) \dk.
  \end{displaymath}
  Taking the formal derivative with respect to $k$ we obtain
  \begin{displaymath}
    \D \hat{f}(k) = - i \int f(x) \otimes x \exp(-i k \cdot x) \dk.
  \end{displaymath}
  If $f \otimes x \in L^1$ then Lebesgue's differentiation theorem can
  be used to make this rigorous. Hence we deduce that $\hat{f} \in
  W^{1,\infty}$. Therefore, since $\hat{f}(0) = 0$, it follows that
  $\hat{f}(k)/|k|$ is bounded as $k \to 0$. Because $f \in L^2$, it
  follows that $\hat{f} \in L^2 \cap L^\infty$, from which is follows
  easily that $\hat{f}(k)/|k| \in L^2$.
\end{proof}

\subsection{The far-field boundary condition}
In this section we interpret the far-field boundary condition
\eqref{eq:farfield_bc_formal} by showing that the space $\He{1}$ is a
natural ansatz space to make this condition rigorous. A simple
motivation for selecting $\He{1}$ as space of functions in which
a solution to \eqref{eq:ellsys_strong}, \eqref{eq:farfield_bc_formal}
is sought, is that this space can be understood as the closure of $\DDe$ in an
``energy-norm''. However, we can give a finer interpretation
of \eqref{eq:farfield_bc_formal} by employing Theorem~\ref{th:growth}.

Let $m = d$. Suppose that an elastic body occupies the reference
domain $\R^d$. Deformations of $\R^d$ are sufficiently smooth
invertible maps $y : \R^d \to \R^d$. Suppose we apply a far-field
boundary condition
\begin{equation}
  \label{eq:farfield_bc_defm}
  y(x) \sim \mA x  \quad \text{ as } |x| \to \infty
\end{equation}
for some non-singular matrix $\mA \in \R^{d \times d}$, which is
usually understood to mean
\begin{displaymath}
  y(x) = \mA x + o(|x|), \quad \text{or,
    equivalently} \quad \frac{|y(x) - \mA x|}{|x|} \to 0 \qquad
  \text{as } |x| \to \infty.
\end{displaymath}

Suppose now that we decompose $y(x) = \mA x + u(x)$; then, the
far-field boundary condition \eqref{eq:farfield_bc_defm} for the
deformation, written in terms of the displacement $u$, becomes
\begin{equation}
  \label{eq:farfield_bc_u_v2}
  u(x) = o(|x|), \quad \text{or, equivalently,} \quad
  \frac{|u(x)|}{|x|} \to 0
  \qquad \text{as } |x| \to \infty.
\end{equation}

While the pointwise condition \eqref{eq:farfield_bc_u_v2} cannot
be satisfied for classes of Sobolev functions,  Theorem
\ref{th:growth} indicates that \eqref{eq:farfield_bc_u_v2}
is satisfied ``on average'' for the representative $J_\infty[u] +
J_0[u]$, for all $[u] \in \He{1}$. Hence it is reasonable to take
$\He{1}$ as the function space in which a solution to
\eqref{eq:ellsys_strong} is sought subject to the far-field displacement
boundary condition \eqref{eq:farfield_bc_formal}.

\subsection{Weak form and well-posedness}
Let $\tC = (\tC_{i\alpha}^{j\beta})_{i,j = 1, \dots, m}^{\alpha, \beta
  = 1, \dots, d} \in (L^\infty)^{m^2d^2}$; we then define the
symmetric bilinear form $a : \He{1} \times \He{1} \to \R$,
\begin{displaymath}
  a(u, v) := \int \tC_{i\alpha}^{j\beta} \D_\alpha u_i \D_\beta
  u_j \dx;
\end{displaymath}
where, here and throughout, we employ the summation
convention. Clearly, $a$ is bounded,
\begin{displaymath}
  a(u, v) \leq c_1 \| \D u \|_{L^2} \| \D v \|_{L^2} \qquad
  \forall u, v \in \He{1},
\end{displaymath}
where $c_1 = \|\tC\|_{L^\infty}$, hence we can pose
\eqref{eq:ellsys_strong}, \eqref{eq:farfield_bc_formal} in weak form:
\begin{equation}
  \label{eq:weak_form}
  a(u, v) = \ell(v) \qquad \forall v \in \DD(\R^d; \R^m),
\end{equation}
where $\ell$ is of the form of Example 1 discussed in
\S~\ref{sec:dual}. An application of the Lax--Milgram theorem gives
the following result.

\begin{theorem}
  \label{th:wellposedness}
  Suppose that $a$ is also coercive:
  \begin{equation}
    \label{eq:coercivity}
    a(u, u) \geq c_0 \| \D u \|_{L^2}^2 \qquad \forall u \in
    \He{1}(\R^d; \R^m),
  \end{equation}
  for some constant $c_0 > 0$; then, \eqref{eq:weak_form} possesses a
  unique solution.
\end{theorem}

\medskip We present three elementary examples of coercivity
\eqref{eq:coercivity}:
\begin{enumerate}
\item If $m = 1$ and $\mC := (\tC_\alpha^\beta)_{\alpha, \beta = 1,
    \dots, d}$ is uniformly positive definite, i.e., $k^{\rm T} \mC(x) k
  \geq c_0 |k|^2$ for a.e.  $x \in \R^d$ and for all $k \in \R^d$,
  then \eqref{eq:coercivity} holds.

\item If $m \in \N$, $\tC$ is a constant tensor and satisfies the
  Legendre--Hadamard condition
\begin{equation}
  \label{eq:LH}
  \tC_{i\alpha}^{j\beta} v_i v_j k_\alpha k_\beta \geq c_0 |v|^2
  |k|^2 \qquad \forall v \in \R^m, \quad k \in \R^d,
\end{equation}
then \eqref{eq:coercivity} holds.

This result is classical if $u \in \DD$. Since $a$ is translation
invariant it follows that it also holds for all $u \in \DDe$. Since
$a$ is bounded and $\DDe$ is dense in $\He{1}$ coercivity holds also
in the full space $\He{1}$.

\item Let $\bar{\tC} \in \R^{d^2m^2}$ be a constant tensor satisfying
  \eqref{eq:LH} with $c_0 = \bar{c}_0$; then, \eqref{eq:coercivity}
  holds with $c_0 = {\bar{c}}_0 - c_1 \| \bar{\tC} - \tC \|_{L^\infty}$.
\end{enumerate}

\begin{remark}
  One may give more general conditions for coercivity of $a$ (or inf-sup conditions) 
  based on G\aa rding's inequality and conditions on the $L^2$-spectrum of $a$.
\end{remark}

\subsection{Regularity}
Higher regularity of the right-hand side leads to higher
regularity of the solution to \eqref{eq:weak_form}. For $s \in \{2, 3,
\dots\}$ we define
\begin{displaymath}
  \He{s} := \b\{ u \in \He{1} \bsep \D u \in H^{s-1} \b\}.
\end{displaymath}
In the following theorem we present conditions for $\He{2}$ and
$\He{3}$ regularity. Regularity in $\He{s}$ for $s \geq 4$ can be
established similarly.

\begin{theorem}
  \label{th:regularity}
  Let all conditions of Theorem \ref{th:wellposedness} be satisfied,
  and let $u$ denote the unique solution to \eqref{eq:weak_form}.
  \begin{enumerate}\ilist
  \item Suppose, in addition, that $\tC \in C^{1}$ and $f \in L^2$;
    then, $u \in \He{2}$ and
    \begin{displaymath}
      \| \D^2 u \|_{L^2} \leq C \b( \| f \|_{L^2} + c_2 \| \D \tC
      \|_{L^\infty} \b).
    \end{displaymath}
  \item Suppose, in addition, that $\tC \in C^2$, $\D\tC \in L^2$, and
    $f \in H^1$; then, $u \in \He{3}$ and
    \begin{displaymath}
      \| \D^3 u \|_{L^2} \leq C \b( \| \D f \|_{L^2} +
      \| \D \tC \|_{L^2} \| \D^2 u \|_{L^2}
      + c_2 \| \D^2 \tC \|_{L^\infty} \b).
    \end{displaymath}
  \end{enumerate}
\end{theorem}
\begin{proof}
  Let $u_* := J_\infty[u] + J_0[u]$ denote a concrete representative
  of the solution $[u] \in \He{1}$ of the weak form
  \eqref{eq:weak_form}. Then clearly $u_*$ satisfies
  \eqref{eq:weak_form}.  The finite difference technique \cite[Section
  6.3.1]{evans} ensures that $u_* \in H^{s+2}_\loc$, and in particular
  $\D u_* \in H^{s+1}_\loc$. The latter property is independent of the
  representative; hence we may say that $\D u \in H^{s+1}_\loc$.

  To obtain the global bound {\it (i)}, we test \eqref{eq:weak_form}
  with $v' = \D_\gamma v$ for some $v \in \DD$, $\gamma \in \{1,\dots,
  d\}$. Then,
  \begin{align*}
    \int f \cdot \D_\gamma v \dx =~& \int \tC_{i\alpha}^{j\beta}
    \D_\alpha u_i \D_\beta \D_\gamma v_j \dx \\
    =~& - \int \B( \D_\gamma \tC_{i\alpha}^{j\beta} \D_\alpha u_i +
    \tC_{i\alpha}^{j\beta} \D_\alpha \D_\gamma u_i\B) \D_\beta v_j \dx,
  \end{align*}
  which implies that
  \begin{displaymath}
    c_0 \| \D(\D_\gamma u) \|_{L^2} \leq \| f \|_{L^2} + \| \D \tC
    \|_{L^\infty} \| \D u \|_{L^2} \leq \| f \|_{L^2} + c_2 \| \D \tC
    \|_{L^\infty}.
  \end{displaymath}

  To prove {\it (ii)} we test with $v' = \D_\gamma \D_\delta v$ and
  perform a similar calculation.
\end{proof}

\bibliographystyle{siam}
\bibliography{biblio}

\end{document}